\documentclass[oneside,11pt]{amsart}
\usepackage{graphicx} 
\usepackage[margin=1in]{geometry}
\usepackage{amsmath, amssymb, amsfonts, amstext, amsthm, amscd}
\usepackage[mathscr]{euscript}
\usepackage{enumerate}
\usepackage{tikz,color,soul}
\usepackage[english]{babel}
\usepackage{chngcntr}
\usetikzlibrary{arrows.meta, positioning,arrows}

\newtheorem{theorem}{Theorem}[section]

\newtheorem{proposition}[theorem]{Proposition}
\newtheorem{lemma}[theorem]{Lemma}
\newtheorem{corollary}[theorem]{Corollary}
\theoremstyle{definition}
\newtheorem{definition}[theorem]{Definition}
\newtheorem{example}[theorem]{Example}

\numberwithin{equation}{section}

\title{Two Laplacians for the resistance distance matrix of a graph}
\author{Shivani Tushar Parab}
\address{St Joseph's University,36, Langford Rd, Langford Gardens, Bengaluru, Karnataka 560027}
\email{shivaniparab004@gmail.com}
\author{Raisa DSouza}
\address{St Joseph's University,36, Langford Rd, Langford Gardens, Bengaluru, Karnataka 560027}
\email{raisadsouza@sju.edu.in}
\date{}

\begin{document}

\maketitle
\begin{abstract}
In this paper, we present two new matrices, namely the resistance Laplacian and resistance signless Laplacian matrix of a connected graph. We provide a generalized form of these matrices for different classes of graphs, including the complete graph, complete bipartite graph, and cycle. We investigate the spectral properties of these matrices, analyzing their eigenvalues and eigenvectors. Moreover, we introduce a concept similar to graph energy and define it as the resistance Laplacian energy of a graph and further discuss some bounds for this energy.
\end{abstract}
\section{Introduction}
 Let $G=(V,E)$ be a connected graph of order $n$ with vertices $V=\left\{ 1,2,\cdots ,n \right\}$. The Laplacian and signless Laplacian matrix of $G$ are defined as $L(G)=\mathfrak{D}(G)-A(G)$ and $Q(G)=\mathfrak{D}(G)+A(G)$, respectively. Here $\mathfrak{D}(G)$ is the diagonal matrix of vertex degrees and $A(G)$ is the adjacency matrix of $G$. The sum of distances from vertex $i$ to all the vertices in $G$ is called the transmission of $i$ and is denoted by $Tr(i)$. 
In \cite{aouchiche2013} Aouchiche and Hansen analogously introduced the distance Laplacian and distance signless Laplacian of a graph $G$ as $\mathcal{D}^L= Diag(Tr)-\mathcal{D}(G)$ and $\mathcal{D}^Q= Diag(Tr)+\mathcal{D}(G)$ respectively, where $Diag(Tr)$ denotes the diagonal matrix of vertex transmissions in $G$ and $\mathcal{D}(G)$ denotes the distance matrix of $G$. In their work,  Aouchiche and Hansen derived formulae for the characteristic polynomial of the Distance Laplacian and Distance Signless Laplacian matrices for specific classes of graphs. Furthermore, they established the equivalence between the distance signless Laplacian matrix, distance Laplacian matrix, and the distance spectra for the class of transmission regular graphs (a graph whose every vertex has the same transmission.)

 In the year 2020, Shareifudding Pirzada et al. obtained the distance signless Laplacian spectrum of the joined union of regular graphs \cite{pirzada2020}. As a consequence, they determined the distance signless Laplacian spectrum of the zero divisor graphs of finite commutative rings $\mathbb{Z}_n$ for some values of $n$.  The concept of the distance Laplacian matrix also led to the introduction of the reciprocal distance Laplacian matrix by Ravindra Bapat and Swarup Kumar Panda in \cite{skp2018}. They showed that the largest eigenvalue of this matrix is at most $n$, with equality holding if and only if the complement graph $\Bar{G}$ of $G$ is disconnected.

Let $G$ be a graph with $n$ vertices and $m$ edges. Then the adjacency matrix eigenvalues $\{ \lambda_1,\lambda_2,\cdots,\lambda_n \}$ obey the following  relations:
     \begin{equation}
     \label{energy}
        \sum_{i=1}^n \lambda_i=0 ; \quad \sum_{i=1}^n \lambda_i^2=2 m, 
     \end{equation}
     Jieshan Yang, Lihua You and Ivan Gutman in \cite{yang2013} defined distance Laplacian energy of a connected graph $G$ with the intention to conceive a graph energy like quantity which satisfies the relations (\ref{energy}).
        \begin{definition}
            Let $G$ be a connected graph on $n$ vertices. Let $\left\{\delta_1^L, \delta_2^L, \ldots, \delta_n^L\right\}$ be the eigenvalues of $\mathcal{D}^L(G)$. Then the distance Laplacian energy of $G$ is defined as 
            \[LE_{\mathcal{D}}(G)=\sum_{i=1}^n\left|\xi_i\right| \]
            where $\xi_i= \delta_i^L-\frac{1}{n} \displaystyle\sum_{j=1}^n Tr(j)$.
        \end{definition}
        
        \begin{theorem} {\cite[Lemma 2.1]{yang2013}}
            Let $G$ be a connected graph of order $n$. Then $\displaystyle\sum_{i=1}^n \xi_i=0$ and $\displaystyle\sum_{i=1}^n \xi_i^2=2 S$ where,$$s=\displaystyle\sum_{1 \leq i<j \leq n}\left(d(i,j)\right)^2 \quad \text{and} \quad  S=s+\frac{1}{2} \sum_{i=1}^n\left(Tr(i)-\frac{1}{n} \sum_{j=1}^n Tr(j)\right)^2.$$
        \end{theorem}
Kinkar Chandra Das, Mustapha Aouchiche, and Pierre Hansen provided bounds on Laplacian energy and signless Laplacian energy, characterized graphs for which these bounds are best possible, and established a relation between these energies in \cite{kcdas2018}. 

In addition to the classical notion of distance based on shortest paths, the concept of resistance distance was introduced by Klein and Randić in \cite{klein1993}. Instead of focusing solely on path lengths, resistance distance considers the electrical resistance between vertices, treating the graph as an electrical network. Resistance distance between vertices $i$ and $j$ of a graph $G$ is defined as the effective resistance between the two vertices when each graph edge is replaced by a unit resistor. There are several equivalent ways to define the resistance distance. We describe one such definition which will be used.
\begin{definition}
Let $G$ be a connected graph with $n$ vertices. Let $L$ be the Laplacian matrix and $L^\dagger $ be the Moore Penrose Inverse of $L$ then for any two vertices $i\neq j$ the resistance distance $r(i,j)$ between $i$ and $j$ is defined as
        $r(i,j) = L_{ii}^\dagger +L_{jj}^\dagger -2L_{ij}^\dagger$.
        \newline
       The resistance distance matrix $R(G)$ is an $n\times n$ matrix whose $ij^{th}$ entry is $r(i,j)$ when $i\neq j$ and $0$ otherwise.
    \end{definition}   
The resistance distance forms a metric on graphs. R. B. Bapat and Somit Gupta obtained the resistance distance between any two vertices of a wheel and a fan in \cite{bapat2010}. In \cite{chen2021}, Wuxian Chen and Weigen Yan determined the formula for computing the resistance distance between any two vertices in the vertex-weighted complete multipartite network graph.

The following two results give the relation between resistance distance and the classical distance in a graph.
\begin{theorem}{\cite[Theorem 9.4]{bapat2010graphs}}
\label{lemm1}
Let $G$ be a connected graph and $i,j \in V(G)$. Then $r(i,j) \leq d(i,j)$.
\end{theorem}
\begin{theorem}{\cite[Theorem 9.6]{bapat2010graphs}}
\label{resdist}
    Let $G$ be a directed connected graph with $V(G)=\{1, \ldots, n\} $. Suppose there is a unique path from $i$ to $j$ in $G$ then $r(i,j)=d(i,j)$. In particular in a tree resistance and classical distance coincide.
\end{theorem} 
Analogous to the distance Laplacian and the distance signless Laplacian matrix we introduce the resistance distance Laplacian and resistance distance signless Laplacian. In the next section we formally define these  matrices and compute their spectra for complete graphs, complete bipartite graphs and cylces. In Section 4 we introduce the resistance distance Laplacian energy and discuss some bounds for this energy.

\section{Resistance Laplacian and Signless Laplacian matrices}

\begin{definition}
 The resistance transmission $RTr(v)$ of a vertex $v$ is defined to be the sum of the resistance distances from $v$ to all other vertices in $G$, i.e,
\[{RTr}(v)=\sum_{u \in V} r(u, v)\]
\end{definition}
A connected graph $G$ is said to be $k$- resistance transmission regular if ${RTr}(v)=k$ for every vertex $v \in V$.
\begin{definition}
    Let $G$ be a connected graph. The resistance Laplacian of graph $G$ is defined as the matrix $R^L= Diag(RTr)-R(G)$ where $Diag(RTr)$ denotes the diagonal matrix of vertex resistance transmissions in $G$.
\end{definition}
\begin{definition}
    Let $G$ be a connected graph. The resistance signless Laplacian of graph $G$ is defined as the matrix $R^Q= Diag(RTr)+R(G)$ where $Diag(RTr)$ denotes the diagonal matrix of vertex resistance transmissions in $G$.
\end{definition}

Since $R(G)$ is a symmetric matrix, both $R^L(G)$ and $R^Q(G)$ are symmetric. Moreover, $R^L$ has row sum $0$. Therefore, $0$ is an eigenvalue of $R^L(G)$ with eigenvector having all entries one. Thus, $R^L$ is singular.
If $G$ is a tree then from Theorem \ref{resdist} we have that $r(i,j)=d(i,j)$ and hence for a tree, $R^L(G)=\mathcal{D}^L(G) \hspace{0.5cm} \text{and} \hspace{0.5cm} R^Q(G)=\mathcal{D}^Q(G)$.

 We begin by establishing the general form for $R^L(G)$ and $R^Q(G)$ for a complete graph, complete bipartite graph and cycles and find their spectra. We next show that $R^L(G)$ is positive semidefinite. We first provide some preliminaries.

\begin{definition}
Consider an $n \times n$ matrix
\[A=\left[\begin{array}{cccc}
A_{1,1} & A_{1,2} & \cdots & A_{1, m} \\
A_{2,1} & A_{2,2} & \cdots & A_{2, m} \\
\vdots & \vdots & & \vdots \\
A_{m, 1} & A_{m, 2} & \cdots & A_{m, m}
\end{array}\right]\]
whose rows and columns are partitioned according to a partition $P=\left\{P_1, P_2, \ldots, P_m\right\}$ of $X=\{1,2, \ldots, n\} .$
\newline
The quotient matrix $Q$ is the $m \times m$ matrix whose entries are the average row sums of the blocks $A_{i,j}$ of $A$.
\end{definition}
The partition $P$ is said to be equitable if each block $A_{i,j}$ of $A$ has constant row (and column) sum and in this case the matrix $Q$ is called equitable quotient matrix.
\begin{theorem}{\cite[Theorem 2.3]{equitablequotient}}
\label{equitable quotient}
    If the partition $P$ of $X$ of matrix $A$ is equitable then each eigenvalue of the equitable quotient matrix $Q$ is an eigenvalue of $A$.
\end{theorem}

\begin{lemma}{\cite[Lemma 4.4]{bapat2010graphs}}
 \label{aibj}
     The eigenvalues of the $n \times n$ matrix $aI+bJ$ are $a$ with multiplicity $n-1$ and $a+nb$ with multiplicity $1.$
 \end{lemma}
  \begin{lemma}{\cite[Theorem 1.1.6]{roger2013}}
 \label{evev}
     Let $p(t)$ be a given polynomial of degree $k.$ If $\left(\lambda,x\right)$ is an eigenvalue-eigenvector pair of $A \in M_n, \text { then } \left(p(\lambda), x\right)$ is an eigenvalue-eigenvector pair of $p(A).$
 \end{lemma}
 \begin{theorem}{\cite[Theorem 5.8]{zhang1999}}
\label{circulant matrix}
    Let $C$ be a circulant matrix in the form 
    \[
    C=
 \begin{bmatrix}
c_0 & c_1 & c_2 & \cdots & c_{n-2} & c_{n-1} \\
c_{n-1} & c_0 & c_1 & \cdots & c_{n-3} & c_{n-2} \\
c_{n} & c_1 & c_0 & \cdots & c_4 & c_3 \\
\vdots & \vdots & \vdots & \ddots & \vdots & \vdots \\
c_{n-2} & c_{n-3} & c_{n-4} & \cdots & c_0 & c_{n-1} \\
c_{n-1} & c_{n-2} & c_{n-3} & \cdots & c_1 & c_0
\end{bmatrix} \]
and let $f(\lambda)=c_0+c_1 \lambda+\cdots+c_{n-1} \lambda^{n-1}$. Then the eigenvalues of $C$ are $f(\omega^k)$, \\$k=0,1, \cdots, n-1$ where $\omega$ is an $n^{th}$ primitive root of unity.
\end{theorem}
\subsection{Complete graph}
\begin{theorem}
\label{lcomplete}
For a complete graph $K_n$ the resistance Laplacian matrix is of the form:
       \begin{align*}
           R^L(K_n)=-\frac{2}{n}J+2I
       \end{align*}
       where $J$ is the matrix with all the entries equal to $1$.
\end{theorem}
\begin{proof}
We know that, $L(K_n)=nI-J$.
    Therefore from Lemma \ref{aibj} the eigenvalues of $L(K_n)$ are $n$ with multiplicity $n-1$ and $0$ with multiplicity $1$.
    The eigenvectors of $L(K_n)$ corresponding to the eigenvalue $n$ are $e_i-e_1$ for $2 \leq i \leq n$. We will denote the eigenvector corresponding to the eigenvalue $0$ as $\mathbf{1}$  consisting of all entries as $1$.
To obtain the spectral decomposition of $L(K_n)$, we apply the Gram-Schmidt orthogonalization process to the eigenvectors. Thus the Moore Penrose Inverse of the Laplacian is given by:
       
\[
L^\dagger =\begin{bmatrix}
        \frac{-1}{\sqrt{2}} & \frac{-1}{\sqrt{6}} & \cdots & \frac{-1}{\sqrt{n(n-1)}} & \frac{1}{\sqrt{n}}\\
        \frac{1}{\sqrt{2}} & \frac{-1}{\sqrt{6}} & \cdots & \frac{-1}{\sqrt{n(n-1)}} & \frac{1}{\sqrt{n}}\\
        0 & \frac{2}{\sqrt{6}} & \cdots & \frac{-1}{\sqrt{n(n-1)}} & \frac{1}{\sqrt{n}}\\
        0 & 0 & \cdots & \frac{-1}{\sqrt{n(n-1)}} & \frac{1}{\sqrt{n}}\\
        \vdots & \vdots && \vdots &\vdots \\
        0 & 0 & \cdots & \frac{n-1}{\sqrt{n(n-1)}} & \frac{1}{\sqrt{n}}\\
        
    \end{bmatrix}
    \hspace{0.02cm}
    \begin{bmatrix}
        \frac{1}{n} & 0 & \cdots & 0&0\\
        0 & \frac{1}{n} & \cdots & 0&0 \\
        \vdots & \vdots &\vdots&\ddots &\vdots \\
        0 & 0 &\cdots & \frac{1}{n} & 0\\
        0&0&\cdots &0 &0    
    \end{bmatrix}
    \hspace{0.02cm}
    \begin{bmatrix}
        \frac{-1}{\sqrt{2}} & \frac{-1}{\sqrt{6}} & \cdots & \frac{-1}{\sqrt{n(n-1)}} & \frac{1}{\sqrt{n}}\\
        \frac{1}{\sqrt{2}} & \frac{-1}{\sqrt{6}} & \cdots & \frac{-1}{\sqrt{n(n-1)}} & \frac{1}{\sqrt{n}}\\
        0 & \frac{2}{\sqrt{6}} & \cdots & \frac{-1}{\sqrt{n(n-1)}} & \frac{1}{\sqrt{n}}\\
        0 & 0 & \cdots & \frac{-1}{\sqrt{n(n-1)}} & \frac{1}{\sqrt{n}}\\
        \vdots & \vdots && \vdots &\vdots \\
        0 & 0 & \cdots & \frac{n-1}{\sqrt{n(n-1)}} & \frac{1}{\sqrt{n}}\\
        
    \end{bmatrix} ^ T 
    \]
     \[ \implies L^\dagger = \begin{bmatrix}
    \frac{n-1}{n^2} & \frac{-1}{n^2}& \cdots & \frac{-1}{n^2}\\
      \frac{-1}{n^2} & \frac{n-1}{n^2} & \cdots & \frac{-1}{n^2} \\
      \vdots & \vdots &\ddots &\vdots \\
        \frac{-1}{n^2} & \frac{-1}{n^2}& \cdots & \frac{n-1}{n^2}  
    \end{bmatrix}
    \]
     Hence the resistance distance between vertices $u$ and $v$ is given by 
    $r(u,v)= L_{uu}^\dagger +L_{vv}^\dagger -2L_{uv}^\dagger =\dfrac{2}{n}$.
    The resistance transmission for all $u\in V(G)$ will be
    ${RTr}(u)=\displaystyle\sum_{v \in V} r(u, v)=2-\dfrac{2}{n}$.
    Consequently, we obtain the resistance Laplacian of the complete graph $K_n$ as 
    $R^L(K_n)=-\dfrac{2}{n}J+2I$
 
\end{proof}
\begin{corollary}
\label{lrevalues}
The eigenvalues of $R^L(K_n)$ are $2$ with multiplicity $n-1$ and $0$ with multiplicity $1$.
\end{corollary}
\begin{proof}
    This follows from Theorem \ref{lcomplete} and Lemma \ref{aibj}.

\end{proof}

\begin{theorem}
\label{scomplete}
 For a complete graph $K_n$ the Resistance signless Laplacian matrix is of the form:
       \begin{align*}
           R^Q(K_n)=\frac{2}{n}J+\left(2-\frac{4}{n}\right)I
       \end{align*}
       \end{theorem}
       \begin{proof}
 From the proof of Theorem \ref{lcomplete}, we deduce that the resistance between any two vertices $u$ and $v$ in $V(K_n)$ is determined as $r(u,v) = \dfrac{2}{n}$, and the resistance transmission for all $u \in V(G)$ is given by ${RTr}(u) = 2 - \dfrac{2}{n}$. Using these results, we get $R^Q(K_n)= Diag(RTr)+R(K_n)
    =\dfrac{2}{n}J+\left(2-\dfrac{4}{n}\right)I$. 
 
       \end{proof}
       \begin{corollary}
           The eigenvalues of $R^Q(K_n)$ are $2-\cfrac{4}{n}$ with multiplicity $n-1$ and $4-\dfrac{4}{n}$ with multiplicity $1$.
       \end{corollary}
       \begin{proof}
   This follows from Theorem \ref{scomplete} and Lemma \ref{aibj}.    

\end{proof}

\subsection{Complete bipartite graph}
\begin{theorem}
\label{lbipartite}
    For a complete bipartite graph $K_{p,q}$, the resistance Laplacian matrix is of the form:
 \begin{align*}
           R^L(K_{p,q})=\begin{bmatrix}
               P & Q\\
               R & S
           \end{bmatrix}
       \end{align*}
where 

$\begin{aligned}
& {P}=\left[\frac{2 p}{q}+\frac{p+q-1}{p}\right] I_{p \times p}-\frac{2}{q} J_{p \times p} , \quad {Q}=-\left[\frac{p+q-1}{p q}\right] J_{p \times q} \\
& {R}=-\left[\frac{p+q-1}{p q}\right] J_{q \times p} ,\quad {S}=\left[\frac{2 q}{p}+\frac{p+q-1}{q}\right] I_{q \times q}-\frac{2}{p} J_{q \times q}
\end{aligned}$
\begin{proof}
    The Laplacian of complete bipartite graph is by:
    \[L(K_{p, q})=\left[\begin{array}{cc}
q I_{p \times p} & -J_{p \times q} \\
-J_{q \times p} & p J_{q \times q}
\end{array}\right]\]

Each block of $L(K_{p, q})$ has a constant row sum, leading to the corresponding equitable quotient matrix:
\[Q=\left[\begin{array}{cc}
q & -q \\
-p & p
\end{array}\right] .\]
The eigenvalues of $Q$ are $0$ and $p+q$. Therefore, according to Theorem \ref{equitable quotient}, we deduce that two of the eigenvalues of $L(K_{p, q})$ are $0$ and $p+q$ each with a multiplicity of $1$. The eigenvector corresponding to the eigenvalue $p+q$ is characterized by the vector with entries $\dfrac{-q}{2}$ for the first $p$ positions and $1$ for the remaining entries. On the other hand, the eigenvector associated with the eigenvalue $0$ is $\mathbf{1}$. Moreover, the remaining two eigenvalues of $L(K_{p, q})$ are $q$ and $p$, with multiplicities of $p-1$ and $q-1$, respectively. The eigenvectors corresponding to the eigenvalue $q$ are $e_i-e_1$ for $2 \leq i \leq p$. Similarly, the eigenvectors associated with the eigenvalue $p$ are denoted as $e_i-e_{p+1}$ for $p+2 \leq i \leq p+q$. Using Gram Schmidt Orthogonalization we obtain the Moore Penrose Inverse of $L(K_{p, q})$ as

\[ L^\dagger =\begin{bmatrix}
    P'&Q'\\
    R' & S'
\end{bmatrix}\]

where

$\begin{aligned}
& {P'}=\frac{1}{q} I+\left(\frac{q^2}{p q(p+q)^2}-\frac{1}{p q}\right) J_{p \times p}, \quad {Q'}=-\left(\frac{pq}{pq (p+q)^2}\right) J_{p \times q} \\
& {R'}=-\left(\frac{pq}{pq (p+q)^2}\right) J_{q \times p}, \quad {S'}=\frac{1}{p} I+\left(\frac{p^2}{p q(p+q)^2}-\frac{1}{p q}\right) J_{q \times q}
\end{aligned}$

\vspace{0.5cm}
Consequently, the resistance distance matrix of the Laplacian is given by
\[R(K_{p,q})=\begin{bmatrix}
    P''&Q''\\
    R'' & S''
\end{bmatrix}\]
where

$\begin{aligned}
& P''=-\frac{2}{q} I+\frac{2}{q}J_{p \times p}, \quad Q''=\left(\frac{p+q-1}{pq}\right) J_{p \times q} \\
& R''=\left(\frac{p+q-1}{pq}\right) J_{q \times p}, \quad S''=-\frac{2}{p} I+\frac{2}{p} J_{q \times q}
\end{aligned}$

The resistance transmission for $u\in V(G)$ will be as follows
\begin{align*}
	RTr(u) = \left\lbrace\begin{array}{cc}
		\dfrac{2(p-1)}{q}+\dfrac{p+q-1}{p}, & 1 \leq u \leq p\\
		\dfrac{2(q-1)}{p}+\dfrac{p+q-1}{q}, & p+1 \leq u \leq p+q
	\end{array}\right. 
\end{align*} 

Hence the Resistance Laplacian of Complete Bipartite Graph is given by:
 \[R^L(K_{p,q})= Diag(RTr)-R(K_{p,q})\]
  \begin{align*}
           R^L(K_{p,q})=\begin{bmatrix}
               P & Q\\
               R & S
           \end{bmatrix}
       \end{align*}
where 

$\begin{aligned}
& {P}=\left[\frac{2 p}{q}+\frac{p+q-1}{p}\right] I_{p \times p}-\frac{2}{q} J_{p \times p} , \quad {Q}=-\left[\frac{p+q-1}{p q}\right] J_{p \times q} \\
& {R}=-\left[\frac{p+q-1}{p q}\right] J_{q \times p}, \quad {S}=\left[\frac{2 q}{p}+\frac{p+q-1}{q}\right] I_{q \times q}-\frac{2}{p} J_{q \times q}
\end{aligned}$

\end{proof}
\begin{corollary}
    The spectrum of $ R^L(K_{p,q})$ is given as :
    \begin{align*}
     \left\lbrace\begin{array}{ccccccc}
		0 && \dfrac{(p+q)^2-p-q}{p q}& & \dfrac{2 p}{q}+\dfrac{p+q-1}{p} && \dfrac{2 q}{p}+\dfrac{p+q-1}{q}\\
		1 && 1 && p-1 && q-1
	\end{array}\right\rbrace. 
 \end{align*}
\end{corollary}
\begin{proof}
    From Theorem \ref{lbipartite} we can deduce that each block of $R^L(K_{p,q})$ has constant row sum. Consequently, the corresponding equitable quotient matrix can be expressed as
\[Q=\left[\begin{array}{cc}
\frac{p+q-1}{p} & -\left(\frac{p+q-1}{p}\right) \\
-\left(\frac{p+q-1}{q}\right) & \frac{p+q-1}{q}
\end{array}\right]\]

The eigenvalues of $Q$ are $0$ and $\dfrac{(p+q)^2-p-q}{p q}$ with each having multiplicity $1$.
Thus, by applying Theorem \ref{equitable quotient}, we can deduce that $R^L(K_{p,q})$ possesses two eigenvalues: $0$ and $\dfrac{(p+q)^2-p-q}{p q}$.
Consider the vector $e_i-e_1$ for $2\leq i \leq p$. Then, 
     \[
     R^L(K_{p,q}) \cdot \left(e_i-e_1\right) =\left( \dfrac{2 p}{q}+\dfrac{p+q-1}{p}\right)\cdot \left(e_i-e_1\right)\hspace{0.4cm}
     \text{for all} \hspace{0.2cm} 2\leq i \leq p
\]
Similarly if we consider vectors $e_i-e_{p+1}$ for all $p+2\leq i \leq p+q$. Then,
     \[
     R^L(K_{p,q}) \cdot \left(e_i-e_{p+1}\right) =\left( \dfrac{2 q}{p}+\dfrac{p+q-1}{q}\right)\cdot \left(e_i-e_{p+1}\right)\hspace{0.4cm}
     \text{for all} \hspace{0.2cm} p+2\leq i \leq p+q
\]
Thus, the spectrum of $ R^L(K_{p,q})$ is given as :
    \begin{align*}
     \left\lbrace\begin{array}{ccccccc}
		0 && \dfrac{(p+q)^2-p-q}{p q}& & \dfrac{2 p}{q}+\dfrac{p+q-1}{p} && \dfrac{2 q}{p}+\dfrac{p+q-1}{q}\\
		1 && 1 && p-1 && q-1
	\end{array}\right\rbrace. 
 \end{align*}

\end{proof}
\begin{theorem}
\label{sbipartite}
    For a complete bipartite graph $K_{p,q}$, the resistance Signless Laplacian matrix is of the form:
 \begin{align*}
           R^Q(K_{p,q})=\begin{bmatrix}
               A & B\\
               C & D
           \end{bmatrix}
       \end{align*}
where 

$\begin{aligned}
&A=\left[\dfrac{2(p-2)}{q}+\dfrac{p+q-1}{p}\right]I_{p \times p}+\dfrac{2}{q}J_{p \times p}, \quad B=\left[\dfrac{p+q-1}{pq}\right]J_{p \times q}\\
&C=\left[\dfrac{p+q-1}{pq}\right]J_{q \times p}, \quad D=\left[\dfrac{2(q-2)}{p}+\dfrac{p+q-1}{q}\right]I_{q \times q}+\dfrac{2}{p}J_{q \times q}
\end{aligned}$
\end{theorem}

\end{theorem}
\begin{proof}
    From the proof of Theorem \ref{lbipartite} we know that the resistance distance matrix of complete bipartite graph is given by: 
\[R(K_{p,q})=\begin{bmatrix}
    P''&Q''\\
    R'' & S''
\end{bmatrix}\]
where

$\begin{aligned}
& P''=-\frac{2}{q} I+\frac{2}{q}J_{p \times p}, ]\quad Q''=\left(\frac{p+q-1}{pq}\right) J_{p \times q} \\
& R''=\left(\frac{p+q-1}{pq}\right) J_{q \times p}, \quad S''=-\frac{2}{p} I+\frac{2}{p} J_{q \times q}
\end{aligned}$

and the resistance transmission for $u\in V(K_{p,q})$ is:
\begin{align*}
	RTr(u) = \left\lbrace\begin{array}{cc}
		\dfrac{2(p-1)}{q}+\dfrac{p+q-1}{p}, & 1 \leq u \leq p\\
		\dfrac{2(q-1)}{p}+\dfrac{p+q-1}{q}, & p+1 \leq u \leq p+q
	\end{array}\right. 
\end{align*} 

Therefore the resistance Laplacian of complete bipartite graph is given by:
 \[R^Q(K_{p,q})= Diag(RTr)+R(K_{p,q})\]
 \begin{align*}
           \implies R^Q(K_{p,q})=\begin{bmatrix}
               A & B\\
               C & D
           \end{bmatrix}
       \end{align*}
where 

$\begin{aligned}
&A=\left[\dfrac{2(p-2)}{q}+\dfrac{p+q-1}{p}\right]I_{p \times p}+\dfrac{2}{q}J_{p \times p}, \quad B=\left[\dfrac{p+q-1}{pq}\right]J_{p \times q}\\
&C=\left[\dfrac{p+q-1}{pq}\right]J_{q \times p}, \quad D=\left[\dfrac{2(q-2)}{p}+\dfrac{p+q-1}{q}\right]I_{q \times q}+\dfrac{2}{p}J_{q \times q}
\end{aligned}$

\end{proof}
\begin{corollary}
    The spectrum of $ R^Q(K_{p,q})$ is given as :
    \begin{footnotesize}
    \begin{align*}
     \left\lbrace\begin{array}{cccccc}
		\frac{5 p^2+(2 p-5) q+5 q^2 \pm\left[\sqrt{9 p^2-14 p q+9 q^2(p+q-1)}\right]-5 p}{2 p q } && 2\left(\frac{p-2}{q}\right)+\frac{p+q-1}{p}& & 2\left(\frac{q-2}{p}\right)+\frac{p+q-1}{q} \\
		1 && p-1 && q-1
	\end{array}\right\rbrace. 
 \end{align*}
 \end{footnotesize}
\end{corollary}
\begin{proof}
    From Theorem \ref{sbipartite} we can deduce that each block of $R^Q(K_{P,Q})$ has constant row sum. Therefore the equitable quotient matrix corresponding to it will be:
\[Q=\left[\begin{array}{cc}
\frac{2(p-2)}{q}+\frac{2(p-1)}{q}+\frac{p+q-1}{p} & \frac{p+q-1}{p} \\
\frac{p+q-1}{q} & \frac{2(q-2)}{p}+\frac{2(q-1)}{p}+\frac{p+q-1}{q}
\end{array}\right]\]

The eigenvalues of $Q$ are 
\[\dfrac{5 p^2+(2 p-5) q+5 q^2 \pm\left[\sqrt{9 p^2-14 p q+9 q^2(p+q-1)}\right]-5 p}{2 p q }\] with multiplicity $1$. Therefore from Theorem \ref{equitable quotient} these two are the eigenvalues of $R^L(K_{p,q})$.
Consider the vector $e_i-e_1$ for $2\leq i \leq p$. Then,
     \[
     R^Q(K_{p,q}) \cdot \left(e_i-e_1\right) = \left[2\left(\frac{p-2}{q}\right)+\frac{p+q-1}{p}\right] \cdot \left(e_i-e_1\right)\hspace{0.4cm}
     \text{for all} \hspace{0.2cm} 2\leq i \leq p
\]
Similarly if we consider vectors $e_i-e_{p+1}$ for all $p+2\leq i \leq p+q$. Then,
     \[
     R^Q(K_{p,q}) \cdot \left(e_i-e_{p+1}\right) =\left[2\left(\frac{q-2}{p}\right)+\frac{p+q-1}{q}\right] \cdot \left(e_i-e_{p+1}\right)\hspace{0.4cm}
     \text{for all} \hspace{0.2cm} p+2\leq i \leq p+q
\]

Thus, the spectrum of $ R^Q(K_{p,q})$ is given as :
\begin{footnotesize}
    
    \begin{align*}
     \left\lbrace\begin{array}{cccccc}
		\frac{5 p^2+(2 p-5) q+5 q^2 \pm\left[\sqrt{9 p^2-14 p q+9 q^2(p+q-1)}\right]-5 p}{2 p q } && 2\left(\frac{p-2}{q}\right)+\frac{p+q-1}{p}& & 2\left(\frac{q-2}{p}\right)+\frac{p+q-1}{q} \\
		1 && p-1 && q-1
	\end{array}\right\rbrace. 
 \end{align*}
 \end{footnotesize}

\end{proof}
\subsection{Cycle}
\begin{lemma}
\label{resi}
    Let $G$ be a connected graph and $R(G)$ be the resistance matrix of $G$. Let $\gamma_1 \geq \gamma_2 \geq \cdots \geq \gamma_n$ be the eigenvalues of $R(G)$. Suppose $G$ is a $k$- resistance transmission regular then we have
    \newline
        (\romannumeral 1) The eigenvalues of $R^L(G)$ are $k-\gamma_n \geq k-\gamma_{n-1} \geq \cdots k-\gamma_1$.
        \newline
        (\romannumeral 2) The eigenvalues of $R^Q(G)$ are $k+\gamma_1 \geq k+\gamma_2 \geq \cdots k+\gamma_n$.
    
\end{lemma}
\begin{proof}
 (\romannumeral 1)   We have $RTr(v)=k$ for all $v \in G.$
    Therefore  
    \[\begin{aligned}
R^L(G) & ={Diag}(R T r)-R(G) ={Diag}\left(k, k, \cdots, k\right)-R(G) =k I-R(G)
\end{aligned}\]
Now the eigenvalues of $R(G)$ are $\gamma_1 \geq \gamma_2 \geq \cdots \gamma_n$.
Therefore the eigenvalues of $-R(G)$ are $-\gamma_n \geq -\gamma_{n-1} \geq \cdots -\gamma_1$.
Consider $p(t)=k-t$. Then by Lemma \ref{evev} we have eigenvalues of $p(R(G))=kI-R(G)$ as 
$p(\gamma_1)=k-\gamma_1,p(\gamma_2)=k-\gamma_2,\cdots,p(\gamma_n)=k-\gamma_n$.
Therefore the eigenvalues of $R^L(G)$ are $k-\gamma_n \geq k-\gamma_{n-1} \geq \cdots k-\gamma_1$.

(\romannumeral 2) Consider $p(t)=k+t$. Then by Lemma \ref{evev} we have eigenvalues of $p(R(G))=kI+R(G)$ as 
$p(\gamma_1)=k+\gamma_1,p(\gamma_2)=k+\gamma_2,\cdots,p(\gamma_n)=k+\gamma_n$.
Therefore The eigenvalues of $R^Q(G)$ are $k+\gamma_1 \geq k+\gamma_2 \geq \cdots k+\gamma_n$.

\end{proof}

\begin{theorem}
\label{lcycle}
The resistance Laplacian matrix of a cycle on $n$ vertices is given by,
\[ R^L(C_n)=\operatorname{circ}\left(\frac{n^2-1}{6}, \frac{-(n-1)}{n}, \frac{-2(n-2)}{n}, \ldots, \frac{-(n-1)(n-(n-1))}{n}\right).\]
\end{theorem}
\begin{proof}
If we consider any two vertices in a cycle say $i$ and $j$ then there are two paths to go from $i$ to $j$, one in clockwise direction and other in anticlockwise direction. From the definition of resistance for $i,j \in V(G)$ we have
\[\begin{aligned}
a_{i (k+1+(i-1))} & =r(i, k+1+(i-1)) =\frac{\left|i-(k+1+(i-1))\right| [n-\left| i-(k+1+(i-1))\right|]}{n} \\
& =\frac{|i-k-1-i+1|[n-|i-k-1-i+1|]}{n}  =\frac{|-k|[n-|-k|]}{n} \\
& =\frac{k[n-k]}{n}  =a_{1(k+1)}
\end{aligned}\]
Now for $i=1$ and $1 \leq j \leq n$ we have 
\[r(1, j)=\left[0, \frac{n-1}{n}, \frac{2(n-2)}{n}, \ldots, \frac{k(n-k)}{n}, \ldots, \frac{(n-1)\left(n-(n-1)\right)}{n}\right].\]
Let $a_{1 (k+1)}=\frac{k(n-k)}{n}$. Then for $i^{th}$ row we will have 
\[\begin{aligned}
a_{i (k+1+(i-1))} & =r(i, k+1+(i-1)) =\frac{\left|i-(k+1+(i-1))\right| [n-\left| i-(k+1+(i-1))\right|]}{n} \\
& =\frac{|-k|[n-|-k|]}{n} =\frac{k[n-k]}{n}  =a_{1(k+1)}
\end{aligned}\]
Thus the resistance matrix of a cycle is given by the circulant matrix:
\begin{equation}
    \label{circulant}
    R(C_n)=\operatorname{circ}\left(0, \frac{n-1}{n}, \frac{2(n-2)}{n}, \ldots, \frac{(n-1)(n-(n-1))}{n}\right).
\end{equation}
For a vertex $v \in V(G)$ the resistance transmission of $v$ will be given by:
\[\begin{aligned}
RTr(v) & =\frac{n-1}{n}+\frac{2(n-2)}{n}+\cdots+\frac{(n-1)[n-(n-1)]}{n} \\
& =\frac{(n-1)(n)}{2}-\frac{1}{n}\left[\frac{(n-1)(n)(2 n-1)}{6}\right]  =\frac{n^2-1}{6}
\end{aligned}\]
Hence,
\begin{equation}
\label{tarns}
    RTr(v)=\frac{n^2-1}{6}, \text{for all \hspace{0.1cm}} v \in V(G)
\end{equation}
From equations \ref{circulant} and \ref{tarns} the resistance Laplacian of cycle will be given as:
\[\begin{aligned}
R^L(C_n) & = Diag(RTr)-R(C_n)\\
& =\frac{n^2-1}{6}I-\operatorname{circ}\left(0, \frac{n-1}{n}, \frac{2(n-2)}{n}, \ldots, \frac{(n-1)(n-(n-1))}{n}\right)\\
& =\operatorname{circ}\left(\frac{n^2-1}{6}, \frac{-(n-1)}{n}, \frac{-2(n-2)}{n}, \ldots, \frac{-(n-1)(n-(n-1))}{n}\right)
\end{aligned}\]

\end{proof}
\begin{theorem}
\label{scycle}
    The resistance signless Laplacian matrix of a cycle on $n$ vertices is given by
\[ R^Q(C_n)=\operatorname{circ}\left(\frac{n^2-1}{6}, \frac{n-1}{n}, \frac{2(n-2)}{n}, \ldots, \frac{(n-1)(n-(n-1))}{n}\right).\]
\end{theorem}
\begin{proof}
    From equations \ref{circulant} and \ref{tarns} the resistance Laplacian of cycle will be given as:
\[\begin{aligned}
R^Q(C_n) & = Diag(RTr)+R(C_n)\\
& =\frac{n^2-1}{6}I+\operatorname{circ}\left(0, \frac{n-1}{n}, \frac{2(n-2)}{n}, \ldots, \frac{(n-1)(n-(n-1))}{n}\right)\\
& =\operatorname{circ}\left(\frac{n^2-1}{6}, \frac{n-1}{n}, \frac{2(n-2)}{n}, \ldots, \frac{(n-1)(n-(n-1))}{n}\right).
\end{aligned}\]

\end{proof}
\begin{corollary}
     Let $g(\lambda)=\dfrac{n-1}{n} \lambda+\dfrac{2(n-2)}{n} \lambda^2 + \cdots+\dfrac{(n-1)(n-(n-1))}{n} \lambda^{n-1}$. Then the eigenvalues of $R^L(C_n)$ are given by $h -g\left(\omega^k\right)$ and that of $R^Q(C_n)$ are given by $h+g\left(\omega^k\right) $ for  $k=0,1, \ldots, n-1$ where $\omega$ is an $n^{th}$ primitive root of unity and $h=\dfrac{n^2-1}{6}$.
\end{corollary}
\begin{proof}
The proof follows from Theorem \ref{lcycle}, Theorem \ref{scycle}, Theorem \ref{circulant matrix} and Lemma \ref{resi}.
   
\end{proof}

\subsection{Some Auxiliary Results}

\begin{definition}
    A matrix $A=\left[a_{i j}\right] \in M_n$ is called diagonally dominant if
    \[\left|a_{i i}\right| \geq \sum_{j \neq i}\left|a_{i j}\right| \quad \text { for all } \quad i=1, \ldots, n\]
\end{definition} 

\begin{proposition}{\cite[Theorem 6.1.10]{roger2013}}
\label{diad dom}
    Suppose that $A=\left[a_{i j}\right] \in M_n$ is symmetric and diagonally dominant. If $a_{ii} \ge 0$ for all $i=1,2,\cdots,n$ then $A$ is positive semidefinite matrix.
\end{proposition}
\begin{proof}
    Now for a vector $x$, $x^T A x=\displaystyle\sum_{i=1}^n a_{i i} x_i^2+\sum_{i \neq j} a_{i j} x_i x_j$. Now since $a_{ii} \ge 0$ we get
    \[\begin{aligned}
x^T A x & =\sum_{i=1}^n\left|a_{i i}\right| x_i^2+\sum_{i \neq j} a_{i j} x_i x_j \\
& \geqslant \sum_{i=1}^n \sum_{j \neq i}\left|a_{i j}\right| x_i^2+\sum_{i \neq j} a_{ij} x_i x_j \geqslant \sum_{i=1}^n \sum_{j \neq i}\left|a_{i j}\right| x_i^2-\sum_{i \neq j}\left|a_{i j}\right|\left|x_i\right|\left|x_j\right|\\
\end{aligned}\]
Since $A$ is symmetric we get
\[\begin{aligned}
x^T A x 
& \ge \sum_{j>1}\left|a_{i j}\right|\left(x_i^2+x_j^2\right)-\sum_{j>i}\left|a_{i j}\right|\left|x_i\right|\left|x_j\right|
=\sum_{j>i}\left|a_{i j}\right|\left(x_i^2+x_j^2-2\left|x_i\right|\left|x_j\right|\right) 
\end{aligned}\]
Therefore we get $x^T A x \ge \displaystyle\sum_{j=1}\left|a_{i j}\right|\left(x_i+x_j\right)^2 \ge 0$. Hence $A$ is positive semidefinite.

\end{proof}

\begin{theorem}
\label{posdef}
Let $G$ be a connected graph. Then $R^L(G)$ is positive semidefinite matrix.

\end{theorem}
\begin{proof}
We have seen that $R^L(G)$ is a symmetric matrix and we have $r(i,j) \ge 0 $ for all $i,j \in V(G)$. Also for a vertex $i \in V(G)$ we have ${RTr}(i)=\displaystyle\sum_{j \in V} r(i, j)$. This implies that $RTr(i) \ge 0 $ for all $i \in V(G)$. Also the $i^{th}$ row in resistance matrix contains $RTr(i)$ as the diagonal entry and the resistance distances from $i$ to all the other vertices as the off diagonal entry making the $R(G)$ a diagonally dominant matrix. Hence from Proposition \ref{diad dom} we have that $R^L(G)$ is positive semidefinite matrix.

\end{proof}

\begin{definition}
    The spectral radius of a square matrix is defined as the maximum of the absolute values of its eigenvalues.
\newline
    Since $R^L(G)$ is positive semidefinite, the spectral radius of $R^L(G)$ will be its largest eigenvalue.
\end{definition}

\begin{lemma}{\cite[Corollary 4.3.15]{roger2013}} \label{inequality}
    Let $\lambda_1(M) \ge \lambda_2(M) \ge \cdots \ge \lambda_n(M)$ be the eigenvalues of arbitrary symmetric matrix $M$. Suppose $C$ is a matrix of order $n$ such that $C=A+B$. Then, \[\lambda_i(A)+\lambda_1(B) \ge \lambda_i(C) \ge \lambda_i(A)+\lambda_n(B) \]
\end{lemma}

\begin{lemma}\cite{lukovits1999}\label{addedge}
    Let $G$ be a graph and $G'=G+e$ where $e \notin E(G)$. Let $r(i,j)$ and $r'(i,j)$ denote the resistance distances between the vertices $i$ and $j$  in graph $G$ and $G'$ respectively. Then $r(i,j) \ge r'(i,j)$ for $i,j=1,2,\cdots,n.$
\end{lemma}

\begin{theorem}\label{spectr}
    Let $G$ be a connected graph and $G'=G+e$ where $e \notin E(G)$. Let $\gamma_1^L(G), \gamma_2^L(G),\cdots,$ 
    $\gamma_n^L(G)=0$ be the eigenvalues of $R^L(G)$ and $\gamma_1^L(G'), \gamma_2^L(G'),\cdots, \gamma_n^L(G')=0$ be the eigenvalues of $G'$. Then $\gamma_i^L(G) \ge \gamma_i^L(G')$ for $i=1,2,\cdots,n$.
\end{theorem}
\begin{proof}
    Using the result from Lemma \ref{addedge}, it follows that $r(i,j) \geq r'(i,j)$ for $i,j=1,2,\cdots,n$. Consequently, we can deduce that $R^L(G)=R^L(G')+M$, where $M$ is a symmetric, diagonally dominant matrix with non-negative diagonal entries.  From Proposition \ref{diad dom} $M$ is positive semidefinite. Since $\mathbf{1}$ is an eigenvector of both $R^L(G)$ and $R^L(G')$ corresponding to the eigenvalue $0$, we can conclude that $\mathbf{1}$ is also an eigenvector of $M$ corresponding to the eigenvalue $0$. Thus, by applying Lemma \ref{inequality}, we have  $0.$ Therefore from Lemma \ref{inequality} we have $\gamma_i^L(G) \ge \gamma_i^L(G')$ for $i=1,2,\cdots,n$.
    
\end{proof}

\begin{theorem}
    Let $G$ be a connected graph of order $n$. Then the spectral radius of $R^L(G)$ is at least $2.$
\end{theorem}
\begin{proof}
    Consider a connected graph $G$ of order $n$. Let $G_1, G_2, \cdots, G_s$ be the sequence of graphs obtained from $G$ by adding an edge in each step. This procedure will eventually terminate when the complete graph $K_n$ is obtained. We denote the resistance distances between the vertices $i$ and $j$ in graph $G$ and $G'$ as $r(i,j)$ and $r^k(i,j)$, respectively. Then, from Lemma \ref{addedge}, it follows that $r(i,j) \geq r^k(i,j)$ for $i,j=1,2,\cdots,n$. Using Theorem \ref{spectr} we can say that $\gamma_i^L(G) \ge \gamma_i^L(K_n)$ for $i=1,2,\cdots,n$. Consequently, from Corollary \ref{lrevalues} we can say that $\gamma_1^L(G) \ge 2.$
    
\end{proof}

\section{Energy of Resistance Laplacian matrix}
 Let $G$ be a connected graph with eigenvalues of resistance distance matrix as $\{\gamma_1, \gamma_2, \ldots, \gamma_n\}$. The resistance distance energy of a graph was defined by Kinkar Chandra Das et al. in \cite{kinkar2012}. The resistance matrix eigenvalues satisfy the following relations 
 \begin{equation}\label{eqenergy}
  \sum_{i=1}^n \gamma_i=0 \quad \text { and } \quad \sum_{i=1}^n \gamma_i^2=2 f   
 \end{equation}

where
$$
f=\sum_{1 \leq i<j \leq n}\left(r(i,j)\right)^2
$$
Hence the resistance distance energy was defined as $E_R(G)=\displaystyle\sum_{i=1}^n\left|\gamma_i\right|$.
Our main aim is to come up with a graph energy like quantity \cite{xli2012}, which satisfies the relations of the form \ref{eqenergy} for distance Laplacian matrix of a connected graph $G$.

Let $G$ be a connected graph with $n$ vertices. Let $\gamma_1^L,\gamma_2^L,\cdots,\gamma_n^L$ be the eigenvalues of  $R^L(G)$. Define \[\eta_i=\gamma_i^L-\frac{1}{n} \displaystyle\sum_{j=1}^n U_j\] where $U_j=RTr(j)$.
\begin{theorem}\label{sum}
    Let $G$ be a connected graph of order $n$. Then $\displaystyle\sum_{i=1}^n \eta_i=0 $ and
    $\displaystyle\sum_{i=1}^n\eta_i^2=2 F$  where $$f=\sum_{i, j=1}^n\left(r(i,j)\right)^2 \quad \text{and} \quad F=f+\frac{1}{2} \sum_{i=1}^n\left(U_i-\frac{1}{n} \sum_{j=1}^n U_j\right)^2.$$
\end{theorem}
\begin{proof}
We have
\begin{equation}
\label{en}
    \sum_{i=1}^n \gamma_i^L=trace\left(R^L\right)=\sum_{i=1}^n U_i
\end{equation}
Therefore 

$\begin{aligned}
\sum_{i=1}^n\left(\gamma_i^L\right)^2 & =trace\left[\left(R^L\right)^2\right]  =\sum_{i=1}^n U_i^2+\sum_{i, j=1}^n\left(\gamma_{i j}\right)^2=\sum_{i=1}^n U_i^2+2 f
\end{aligned}$
where $f=\displaystyle\sum_{i, j=1}^n\left(r(i,j)\right)^2$

Hence we get that $
\displaystyle\sum_{i=1}^n \eta_i  =\displaystyle\sum_{i=1}^n\left(\gamma_i^L-\frac{1}{n} \displaystyle\sum_{j=1}^n U_j\right) =\displaystyle\sum_{i=1}^n \gamma_i^L-\sum_{j=1}^2 U_j$.
From equation \ref{en} it follows that $\displaystyle\sum_{i=1}^n \eta_i=0$.
\newline
Next we have 
\[\begin{aligned}
\sum_{i=1}^n \eta_i^2 & =\sum_{i=1}^n\left(\gamma_i^L-\frac{1}{n} \sum_{j=1}^n U_j\right)^2 =\sum_{i=1}^n\left(\gamma_i^L\right)^2-\frac{2}{n} \sum_{j=1}^n U_j \sum_{i=1}^n \gamma_i^L+\frac{1}{n^2}\left(\sum_{j=1}^n U_j\right)^2 \\
& =\sum_{i=1}^n U_i^2+2 f-\frac{2}{n} \sum_{j=1}^n U_j \sum_{i=1}^n \gamma_i^L+\frac{1}{n^2}\left(\sum_{j=1}^n U_j\right)^2  =2 f+\sum_{i=1}^n\left(U_i-\frac{1}{n} \sum_{j=1}^n U_j\right)^2 =2 F
\end{aligned}\]
where  $F=f+\frac{1}{2} \displaystyle\sum_{i=1}^n\left(U_i-\frac{1}{n} \displaystyle\sum_{j=1}^n U_j\right)^2$. Hence the result.

\end{proof}
\begin{definition}
    Let $G$ be a connected graph on $n$ vertices and let $U_j$ denote the resistance transmission of vertex $j \in V(G)$. Let $\gamma_1^L,\gamma_2^L,\cdots,\gamma_n^L$ be the eigenvalues of  $R^L(G)$. Then we define the resistance Laplacian energy of graph $G$ as: \[LE_R(G)=\sum_{i=1}^n\left|\eta_i\right|\]
    where $\eta_i=\gamma_i^L-\dfrac{1}{n} \displaystyle\sum_{j=1}^n U_j$.
    
\end{definition}

\begin{example}
For a complete graph $K_n$ we know that the eigenvalues of $R^L(K_n)$ are $2$ with multiplicity $n-1$ and $0$ with multiplicity $1$ and $U_j=2-\dfrac{2}{n}$ for all $j \in V(K_n)$. Therefore we have
\[\begin{aligned}
 L E_{R}(K_n) & =\sum_{i=1}^n\left|\gamma_i^L-\frac{1}{n} \sum_{j=1}^n U_j\right| =\sum_{i=1}^n\left|\gamma_i^L-\frac{1}{n} (2 n-2)\right| \\
& =\sum_{i=1}^n\left|\gamma_i^L(K_n)-2+\frac{2}{n}\right|  =\sum_{i=1}^n\left|2- 2+\frac{2}{n}\right|+\left|-2+\frac{2}{n}\right| \\
& =\frac{2}{n}(n-1)+2-\frac{2}{n}  =4\left(1-\frac{1}{n}\right) \\
\end{aligned}\]
\end{example}
\begin{theorem}
    Let $G$ be a $k$- resistance transmission regular graph. Let $\{\gamma_1,\gamma_2,\cdots,\gamma_n\}$ and $\{\gamma_1^L,\gamma_2^L,\cdots,\gamma_n^L\}$ be the sets of eigenvalues of $R(G)$ and $R^L(G)$ respectively. Then $L E_R(G)=E_R(G)$.
\end{theorem}
\begin{proof}
    Since $G$ is a $k$- resistance transmission regular graph we have $k=U_i=\dfrac{1}{n} \displaystyle\sum_{j=1}^n U_j$ for $i=1,2, \ldots, n$. From Lemma \ref{resi} we get
    \[\eta_i=\gamma_i^L-\dfrac{1}{n} \sum_{j=1}^n U_j=\left(k-\gamma_{n+1-i}\right)-k=-\gamma_{n+1-i}\]
    Therefore $LE_R(G)=\displaystyle\sum_{i=1}^n\left|\eta_i\right|=\displaystyle\sum_{i=1}^n\left|-\gamma_{n+1-i}\right|=\displaystyle\sum_{j=1}^n\left|\gamma_{j}\right|=E_R(G)$.

\end{proof}

\begin{theorem}
    Let $G$ be a connected graph of order $n$. Then
    \[2\sqrt{F}\leq LE_R(G) \leq \sqrt{2nF}\]
\end{theorem}
\begin{proof}
    Consider the expression $T=\displaystyle\sum_{i=1}^n\displaystyle\sum_{j=1}^n \left(\left|\eta_i\right|-\left|\eta_j\right|\right)^2$. Expanding the expression we have
    \[T=2 n \sum_{i=1}^n\left|\eta_i\right|^2-2\left(\sum_{i=1}^n\left|\eta_i\right|\right)\left(\sum_{j=1}^n\left|\eta_j\right|\right)=2 n \cdot 2 F-2 L E_R(G)^2=4 n F-2 L E_R(G)^2\]
    Now since $T \geq 0$ we have $4 n F-2 L E_R(G)^2 \ge 0$. Therefore, $LE_R(G) \leq \sqrt{2nF}$ for $F>0$.
    Now we have,
    \[
    \begin{aligned}
2 F & =\sum_{i=1}^n \eta_i^2=\left(\sum_{i=1}^n \eta_i\right)^2-2 \sum_{1 \leq i<j \leq n} \eta_i \eta_j 
\end{aligned}
\]
From Theorem \ref{sum}, $\displaystyle\sum_{i=1}^n \eta_i=0 $. Hence we get
\[
\begin{aligned}
    2F =-2 \sum_{1 \leq i<j \leq n} \eta_i \eta_j=2\left|\sum_{1 \leq i<j \leq n} \eta_i \eta_j\right| \leq 2 \sum_{1 \leq i<j \leq n}\left|\eta_i\right|\left|\eta_j\right| .
\end{aligned}
\]
Thus, 
\[\begin{aligned}
L E_R(G)^2 & =\left(\sum_{i=1}^n\left|\eta_i\right|\right)^2 \\
& =\sum_{i=1}^n\left|\eta_i\right|^2+2 \sum_{1 \leq i<j \leq n}\left|\eta_i\right|\left|\eta_j\right| \geq 2F+2 F=4F 
\end{aligned}\]
  This implies that $2\sqrt{F}\leq LE_R(G).$  
  
\end{proof}
 \begin{theorem}
      Let $G$ be a connected graph of order $n$. Then
      \[L E_R(G) \leq \frac{1}{n} \sum_{i=1}^n U_i+\sqrt{(n-1)\left[2 F-\left(\frac{1}{n} \sum_{i=1}^n U_i\right)^2\right]} \]
  \end{theorem}
  \begin{proof}
      We know that $\gamma_n^L=0$ is an eigenvalue of $R^L(G)$. Therefore $\eta_n=-\dfrac{1}{n}\displaystyle\sum_{i=1}^n U_i$.
      Consider $T'=\displaystyle\sum_{i=1}^{n-1}\displaystyle\sum_{j=1}^{n-1}\left(\left|\eta_i\right|-\left|\eta_j\right|\right)^2$. Expanding the expression we get
  \[
  \begin{aligned}   
  T'&  =2 (n-1) \sum_{i=1}^{n-1}\left|\eta_i\right|^2-2\left(\sum_{i=1}^{n-1}\left|\eta_i\right|\right)\left(\sum_{j=1}^{n-1}\left|\eta_j\right|\right) \\
  & =2(n-1)\left[2 F-\left(\frac{1}{n} \sum_{i=1}^n U_i\right)^2\right]-2\left(L E_R(G)-\frac{1}{n} \sum_{i=1}^n U_i\right)^2 \geq 0 \\
  \end{aligned}
  \]
  Thus, $L E_R(G) \leq \dfrac{1}{n} \displaystyle\sum_{i=1}^n U_i+\sqrt{(n-1)\left[2 F-\left(\dfrac{1}{n} \displaystyle\sum_{i=1}^n U_i\right)^2\right]}$
 
  \end{proof}
   \begin{lemma}{\cite[Corollary 8.1.20]{roger2013}}\label{spectral radius}
      Let $A=[a_{ij} \in M_n]$ be a positive semidefinite matrix with largest eigenvalue $\lambda_1$. Then $\lambda_1 \ge \displaystyle\max_{i} a_{ii}.$ 
  \end{lemma}
   \begin{theorem}
      Let $G$ be a connected graph of order $n$. Let $\eta_1=\gamma_1^L-\dfrac{1}{n} \displaystyle\sum_{j=1}^n U_j$ where $\gamma_1^L$ is the largest eigenvalue of $R^L(G)$. Then 
      \[L E_R(G) \leq \eta_1+\sqrt{(n-1)\left(2 F-\eta_1^2\right)}\]
  \end{theorem}
  \begin{proof} 
  We will renumber the vertices of $G$ so that $U_1=\displaystyle\max_{i \in V(G)} U_i$.
      Consider the two $(n-1)$-dimensional vectors $(1,1,\cdots,1)$ and $(\left|\eta_2\right|,\left|\eta_3\right|,\cdots,\left|\eta_n\right|)$. Then by Cauchy Schwarz inequality we have 
      \[\left(\sum_{i=2}^n\left|\eta_i\right|\right)^2 \leq(n-1) \sum_{i=2}^n\left|\eta_i\right|^2\]
      which implies, $$\left(LE_R(G)-\left|\eta_1\right|\right)^2 \leq(n-1)\left(2 F-\left|\eta_1\right|^2\right).$$
      Therefore we get 
      \begin{equation}\label{energy2}
          L E_R(G) \leq\left|\eta_1\right|+\sqrt{(n-1)\left(2 F-\left|\eta_1\right|^2\right)} 
    \end{equation}
      Now from Lemma \ref{spectral radius} we have $\gamma_1^L \ge D_1$. Thus $\eta_1=\gamma_1^L-\dfrac{1}{n} \displaystyle\sum_{j=1}^n U_j=U_1-\dfrac{1}{n} \displaystyle\sum_{j=1}^n U_j \ge 0.$. Therefore the inequality \ref{energy2} can be written as \[L E_R(G) \leq\eta_1+\sqrt{(n-1)\left(2 F-\eta_1^2\right)}\]
  \end{proof}

\bibliography{refs.bib}

\begin{thebibliography}{10}

\bibitem{aouchiche2013}
M.~Aouchiche and P.~Hansen, ``Two laplacians for the distance matrix of a
  graph,'' {\em Linear Algebra and its applications}, vol.~439, pp.~21--33,
  2013.

\bibitem{pirzada2020}
M.~A. S.~Pirzada, Bilal A.~Rather and T.~A. Chisht, ``On distance signless
  laplacian spectrum of graphs and spectrum of zero divisor graphs of
  $\mathbb{Z}_n$,'' {\em Linear and Multilinear algebra}, vol.~70,
  pp.~3354--3369, 2020.

\bibitem{skp2018}
R.~Bapat and S.~K. Panda, ``The spectral radius of the reciprocal distance
  laplacian matrix of a graph,'' {\em Bulletin of the Iranian Mathematical
  Society}, vol.~44, pp.~1211--1216, 2018.

\bibitem{yang2013}
L.~Y. Jieshan~Yang and I.~Gutman, ``Bounds on the distance laplacian energy of
  a graph,'' {\em Kragujevac Journal of Mathematics}, vol.~37, pp.~245--255,
  2013.

\bibitem{kcdas2018}
M.~A. Kinkar Chandra~Das and P.~Hansen, ``On (distance) laplacian energy and
  (distance) signless laplacian energy of graphs,'' {\em Discrete Applied
  Mathematics}, vol.~243, pp.~172--185, 2018.

\bibitem{klein1993}
D.~Klein and M.~Randic, ``Resistance distance,'' {\em Journal of Mathematical
  Chemistry}, vol.~12, pp.~81--95, 1993.

\bibitem{bapat2010}
R.~Bapat and S.~Gupta, ``Resistance distance in wheels and fans,'' {\em Indian
  Journal of Pure and Applied Mathematics}, vol.~41, pp.~1--13, 2010.

\bibitem{chen2021}
W.~Chen and W.~Yan, ``Resistance distances in vertex-weighted complete
  multipartite graphs,'' {\em Applied Mathematics and Computation}, vol.~409,
  2021.

\bibitem{bapat2010graphs}
R.~Bapat, {\em Graphs and Matrices}.
\newblock Springer, 2010.

\bibitem{equitablequotient}
M.~Y. Lihua~You {\em et~al.}, ``On the spectrum of an equitable quotient matrix
  and its application,'' {\em Linear algebra and its Applications}, vol.~577,
  pp.~21--40, 2019.

\bibitem{roger2013}
C.~J. Roger~Horn, {\em Matrix Analysis}.
\newblock Cambridge, 2013.

\bibitem{zhang1999}
F.~Zhang, {\em Matrix Theory}.
\newblock Springer, 1999.

\bibitem{lukovits1999}
N.~T. I.~Lukovits, S.~Nikolic, ``Resistance distance in regular graphs,'' {\em
  International Journal of Quantum Chemistry}, vol.~71, pp.~217--225, 1999.

\bibitem{kinkar2012}
A.~S.~C. Kinkar Ch.~Das, A. Dilek~Gungor, ``On kirchhoff index and
  resistance–distance energy of a graph,'' {\em Communications in
  Mathematical and in Computer Chemistry}, vol.~67, pp.~541--556, 2012.

\bibitem{xli2012}
I.~G. Xueliang~Li, Yongtang~Shi, {\em Graph Energy}.
\newblock Springer, 2012.

\end{thebibliography}
\bibliographystyle{ieeetr}
\end{document}